\documentclass[11pt]{amsart}

\usepackage{upgreek}
\usepackage{libertine} % para texto
\usepackage{euler} %para matemáticas
\usepackage{eucal} %para caligr. en matemáticas
\usepackage[utf8]{inputenc}

\usepackage[T1]{fontenc}
\usepackage{textcomp}
\usepackage[utf8]{inputenc}

\usepackage{amsopn, amsthm, amsgen, amscd,amsmath,amssymb}
\usepackage[bbgreekl]{mathbbol}
\DeclareMathAlphabet{\mathbbm}{U}{bbm}{m}{n}% from bbm.sty
\usepackage{color}
\usepackage{tikz}
\usepackage{epstopdf}
\usepackage{enumerate}

\author{Alexander Berenstein}
\address{Alexander Berenstein
\\Universidad de los Andes
\\ Departamento de Matem\'aticas
\\Cra 1 \# 18A-10, Bogot\'a, Colombia.}

\urladdr{www.matematicas.uniandes.edu.co/\textasciitilde aberenst}

\author{Tapani Hyttinen}
\address{Tapani Hyttinen
\\University of Helsinki
\\Department of Mathematics and Statistics
\\Gustaf H\"allstr\"ominkatu 2b. Helsinki 00014, Finland.}
\email{tapani.hyttinen@helsinki.fi}

\author{Andr\'es Villaveces}
\address{Andr\'es Villaveces
\\Universidad Nacional de Colombia
\\ Departamento de Matem\'aticas
\\Av. Cra 30 \# 45-03, Bogot\'a 111321, Colombia.}
\email{avillavecesn@unal.edu.co}

\thanks{This work was partially supported by Colciencias grant
  \emph{Métodos de Estabilidad en Clases No Estables.} The second and
  third author were also sponsored by Catalonia's \emph{Centre de
    Recerca Matemàtica} (Intensive Research Program in Strong Logics)
  and by the University of Helsinki in 2016 for part of this work. The
  third author was also partially sponsored by Colciencias 
  (Proy. 1101-05-13605 CT-210-2003).}

\makeatletter
\def\newrefformat#1#2{%
  \@namedef{pr@#1}##1{#2}}
\def\prettyref#1{\@prettyref#1:}
\def\@prettyref#1:#2:{%
  \expandafter\ifx\csname pr@#1\endcsname\relax%
    \PackageWarning{prettyref}{Reference format #1\space undefined}%
    \ref{#1:#2}%
  \else%
    \csname pr@#1\endcsname{#1:#2}%
  \fi%
}
\makeatother

\newrefformat{eq}{\textup{(\ref{#1})}}
\newrefformat{cha}{Chapter \ref{#1}}
\newrefformat{sec}{Section \ref{#1}}
\newrefformat{tab}{Table \ref{#1} on page \pageref{#1}}
\newrefformat{fig}{Figure \ref{#1} on page \pageref{#1}}

\makeatletter
\def\indsym#1#2{%
  \setbox0=\hbox{$\m@th#1x$}%
  \kern\wd0%
  \hbox to 0pt{\hss$\m@th#1\mid$\hbox to 0pt{$\m@th#1^{#2}$}\hss}%
  \lower.9\ht0\hbox to 0pt{\hss$\m@th#1\smile$\hss}%
  \kern\wd0} \newcommand{\ind}[1][]{\mathop{\mathpalette\indsym{#1}}}
\def\nindsym#1#2{%
  \setbox0=\hbox{$\m@th#1x$}%
  \kern\wd0%
  \hbox to 0pt{\mathchardef\nn="3236\hss$\m@th#1\nn$\kern1.4\wd0\hss}
  \hbox to 0pt{\hss$\m@th#1\mid$\hbox to 0pt{$\m@th#1^{#2}$}\hss}%
  \lower.9\ht0\hbox to 0pt{\hss$\m@th#1\smile$\hss}%
  \kern\wd0}
\newcommand{\nind}[1][]{\mathop{\mathpalette\nindsym{#1}}}
\makeatother

\def\bb{\bar b}   \def\bz{\bar z}
  \def\bc{\bar c} 
 \def\ba{\bar a}  \def\be{\bar e}
\def\bd{\bar d}

\def\1{1}
\def\0{0}

  \def\U{\mathcal{U}}

  \def\C{\mathcal{C}} 
   
     \def\Sum{\sum}
\def\bra{\langle} \def\ket{\rangle} \def\vphi{\varphi}
\def\bz{\bar z}       \def\bu{\bar u}

\def\C{\mathcal{A}}

\def\C{\mathcal{C}}

\def\H{\mathcal{H}}
\def\U{\mathcal{U}}

\def\vphi{\varphi}
\def\dotminus{-^{\! \! \! \!  \cdot}}

\def\Astr{\mathfrak{A}}
\def\Bstr{\mathfrak{B}}
\def\Cstr{\mathfrak{C}}

\newcommand{\mynewthm}[3][dummythm]{%
  \newtheorem{#2}[#1]{#3}%
  \newrefformat{#2}{#3 \ref{##1}}%
}

\swapnumbers

\theoremstyle{plain}
\mynewthm{theo}{Theorem}
\mynewthm{prop}{Proposition}
\mynewthm{lema}{Lemma}
\mynewthm{fact}{Fact}
\mynewthm{coro}{Corollary}
\mynewthm{ques}{Question}
\mynewthm{exer}{Exercise}

\theoremstyle{definition}
\mynewthm{defi}{Definition}
\mynewthm{rema}{Remark}
\mynewthm{obse}{Observation}
\mynewthm{conv}{Convention}
\mynewthm{nota}{Notation}
\mynewthm{exam}{Example}

\DeclareMathOperator{\dcl}{dcl}
\DeclareMathOperator{\acl}{acl}
\DeclareMathOperator{\cl}{cl}
\DeclareMathOperator{\bcl}{bcl}
\DeclareMathOperator{\ecl}{ecl}

\DeclareMathOperator{\events}{events}

 \DeclareMathOperator{\tp}{tp}

\DeclareMathOperator{\spane}{span}
\DeclareMathOperator{\dist}{dist}
\DeclareMathOperator{\qftp}{qftp}
\DeclareMathOperator{\pr}{pr}
\DeclareMathOperator{\id}{id}

\title{Hilbert spaces with generic predicates}

\begin{document}
\maketitle
\normalsize

\begin{abstract}
  We study the model theory of expansions of Hilbert spaces by generic
  predicates. We first prove the existence of model companions for
% two kinds of 
  generic  expansions of Hilbert spaces in the form first of a
  distance function to a \emph{random substructure}, then a distance
  to a random subset. The theory obtained
  with the random substructure is $\omega$-stable, while the one
  obtained  with the distance to a random subset is $TP_2$ and
  $NSOP_1$. That example is the first continuous structure in that
  class. 
\end{abstract}

\section{Introduction}
This paper deals with Hilbert spaces expanded with random predicates
in the framework of continuous logic as developed in \cite{BBHU}. The
model theory of Hilbert spaces is very well understood, see
\cite[Chapter 15]{BBHU} or \cite{BB}. However, we briefly
review some of its properties at the end of this section.

%\marginpar{p\'arrafo nuevo julio 14}
In this paper we build several new expansions, by various kinds of
random predicates (random substructure and the distance to a random
subset) of Hilbert spaces, and
study them within the framework of continuous logic. While our
constructions are not exactly metric Fra\"\i ss\'e (failing the
hereditary property), some of them are
indeed amalgamation classes and we study the model theory of their
limits.

Several papers deal with generic expansions of Hilbert
spaces. Ben Yaacov, Usvyatsov and Zadka \cite{BUZ} studied the
expansion of a Hilbert space with a generic automorphism.
The models of this theory are  expansions of Hilbert spaces with a
unitary map whose spectrum is $S^1$. A model of this theory can be
constructed by amalgamating together the collection of $n$-dimensional
Hilbert spaces with a unitary map whose eigenvalues are the $n$-th
roots of unity as $n$ varies in the positive integers. More work on
generic automorphisms can be found in \cite{Be}, where the first
author of this paper studied Hilbert spaces expanded with a random
group of automorphisms $G$.

There are also several papers about expansions of Hilbert spaces with
random subspaces. In \cite{BB} the first author and Buechler
identified the
saturated models of the theory of beautiful pairs of a Hilbert
space. An analysis of lovely pairs (the generalization of beautiful
pairs (belles paires) to simple theories) in the setting of compact
abstract theories is
carried out in \cite{BY2}. In the (very short)
second section of this paper
we build the beautiful pairs of Hilbert spaces as
the model companion of the theory of Hilbert spaces with an orthonormal 
projection. We provide an axiomatization for this class and we show
that the
projection operator into the subspace is interdefinable with a
predicate for the
distance to the subspace. We also prove that the theory of beautiful
pairs of Hilbert spaces is $\omega$-stable. Many of the properties of
beautiful pairs of Hilbert spaces are known from the literature or
folklore, so this section is mostly a compilation of results.

In the third section we add a predicate for the distance to a random
subset.
This construction was inspired by the idea of finding an analogue to
the first order generic predicates studied by Chatzidakis and Pillay
 in \cite{CP}. The
axiomatization we found for the model companion was inspired in
the ideas of  \cite{CP} together with the following observation: in
Hilbert spaces there is a definable function that measures the
distance between a point and a model. We prove that the theory of
Hilbert spaces with a generic predicate is unstable. We also study a
natural notion of independence in a monster model of this theory and
prove some of its properties. Several
natural independence notions have various good properties, but the
theory fails to be simple and even fails to be NTP$_2$.
%The second author provided
%the arguments for this part of the work.

% In a short final section, we describe the case of expansions of
% Hilbert spaces with a random predicate that satisfies a $1$-Lipschitz
% modulus of uniform continuity. We briefly describe the
% model companion of this theory and natural notions of independence in
% a monster model. Additionally, we discuss the import of our results in
% the general quest for understanding of the stability hierarchy in
% continuous logic.

% Finally, in the fifth section we deal with expansions of Hilbert spaces
% with a random predicate that satisfies a $1$-Lipschitz modulus of
% uniform continuity. We find a model companion for such a theory. Again
% the main tool used in the axiomatization is the existence of a
% definable function that measures the distance between a point and a
% model. We also study a natural notion of independence in a monster
% model of this theory and prove some of its properties. As before the
% theory obtained is unstable but it is unclear if the model companion
% is a simple theory.

\subsection{Model theory of Hilbert spaces (quick review)}\label{modthH}

\subsubsection{Hilbert spaces}
We follow \cite{BBHU} in our  study of the model theory of a real Hilbert
space and its expansions. We assume the reader is familiar with the
basic concepts of continuous logic as presented in \cite{BBHU}. A
Hilbert space $\H$ can be seen as a multi-sorted structure
$(B_n(H),0,+,\bra,\ket,\{\lambda_r:r\in \mathbb{R}\})_{0<n<\omega}$,
where $B_n(H)$
is the ball of radius $n$, $+$ stands for addition of vectors
(defined from $B_n(H)\times B_n(H)$ into $B_{2n}(H)$),
$\bra,\ket\colon B_n(H)\times
B_n(H)\to [-n^2,n^2]$ is the inner product, $0$ is a constant for the
zero
%vector and $\lambda_r\colon B_n(H)\to B_{n(\llbracket |r|
vector and $\lambda_r\colon B_n(H)\to B_{n( \lceil |r|\rceil)}H$ is
the multiplication by $r\in \mathbb{R}$.
%  \rrbracket+1)}H$ is the multiplication by $r\in \mathbb{R}$.

We denote by $L$ the language of Hilbert spaces and by $T$ the theory
of Hilbert spaces.

By a universal domain $\H$ of $T$ we mean a Hilbert space $\H$
which is $\kappa$-saturated and $\kappa$-strongly homogeneous with
respect to types in the language $L$, where $\kappa$ is a cardinal
larger
than $2^{\aleph_0}$. Constructing such a
structure is straightforward ---just consider a Hilbert space with
an orthonormal basis of cardinality at least $\kappa$.

We will assume that the reader is familiar with the metric versions of
\emph{definable closure} and \emph{non-dividing}. The reader can
check \cite{BBHU,BB} for the definitions.

\begin{nota}
Let $\dcl$ stand for the definable closure and $\acl$ stand for the
algebraic closure in the language $L$.
\end{nota}

\begin{fact}\label{dclh}
  Let $A\subset \H$ be small. Then $\dcl(A)=\acl(A)=$ the smallest
  Hilbert subspace of $\H$ containing $A$.
\end{fact}

\begin{proof}
See Lemma 3 in \cite[p. 80]{BB}
\end{proof}

%\subsubsection{Non-dividing in $T$}
%First let us
Recall a characterization of non-dividing in pure Hilbert
spaces (that will be useful in the more sophisticated constructions in
forthcoming sections):

\begin{prop}\label{independence}
Let $B,C\subset H$ be small, let $(a_1,\dots,a_n) \in
{H}^n$ and assume that $C=\dcl(C)$, so $C$ is a Hilbert
subspace of $H$. Denote by $P_C$ the projection on $C$. Then
$\tp(a_1,\dots,a_n/C \cup B)$ does not divide over $C$ if and only
if for all $i\leq n$ and all $b\in B$, $a_i-P_C(a_i) \perp
b-P_C(b)$.
\end{prop}

\begin{proof}
Proved as Corollary 2 and Lemma 8 of \cite[pp. 81--82]{BB}.
\end{proof}

For $A,B,C\subset H$ small, we say that $A$ is independent from 
$B$ over $C$ if for all $n\geq 1$ and $\ba\in A^n$, $\tp(\ba/C\cup B)$ does 
not divide over $C$.

Under non-dividing independence, types over sets are stationary. In particular, the independence theorem
holds over sets, and we may refer to this property as $3$-existence. 
It is also important to point out that non-dividing is
\emph{trivial}, that is, for all sets $B,C$ and tuples
$(a_1,\dots,a_n)$ from $H$, $\tp(a_1,\dots,a_n/C\cup B)$ does
not divide over $C$ if and only if $\tp(a_i/B\cup C)$ does not
divide over $C$ for $i\leq n$.

%Finally, let us recall Definition 13 from \cite[p. 83]{BB}. See the
%discussion on the necessity of building canonical bases over Morley
%sequences there.

%\begin{defi}
%Let $(H,+,\bra,\ket,\dots)$ be a stable $\kappa$-universal domain
%of an expansion of a Hilbert space. We say the structure
%$(H,+,\bra,\ket,\dots)$ has \emph{built-in canonical bases} if
%for all $c_1,\dots,c_n\in H$ and small $A\subset H$, there is a
%small set $B\subset H$ such that for any Morley sequence
%$I=((c_1^i,\dots,c_n^i):i\in \omega)$ over $A$ such that
%$(c_1^0,\dots,c_n^0)=(c_1,\dots,c_n)$, the parallelism class of
%$I$ is interdefinable with $B$. That is, $B$ is fixed under an
%automorphism $\Psi$ of $(H,+,\bra,\ket,\dots)$ if and only if the
%parallelism class of $I$ is fixed under $\Psi$. We call such $B$ a
%\emph{built-in canonical base} for the Lascar strong type of
%$(c_1,\dots,c_n)$ over $A$.
%\end{defi}

%Let $C\subset H$ be small, $C=\dcl(C)$ and let
%$(a_1,\dots,a_n) \in H^n$. Then by Theorem 3 in \cite[p. 84]{BB} the set
%$\{P_C (a_i):i\leq n\}$ forms a \emph{built-in
%canonical base} for $\tp(a_1,\dots,a_n/C)$.

\section{Random subspaces and beautiful pairs}

We now deal with the easiest situation: a Hilbert space with an
orthonormal projection operator onto a \emph{subspace}. Let
$L_p=L\cup\{P\}$ where $P$ is a new unary function and we consider
structures of the form $(\H,P)$, where $P\colon H\to H$ is a
projection into a subspace. Note that $P\colon B_n(H)\to B_n(H)$ and
that $P$ is determined by its action on $B_1(H)$. Recall that
projections are bounded linear operators, characterized by two
properties:

\begin{enumerate}
\item $P^2=P$
\item $P^*=P$
\end{enumerate}

The second condition means that for any $u,v\in H$, $\bra P(u),v
\ket=\bra u,P(v) \ket$. A projection  also satisfies,
for any $u,v\in H$, $\|P(u)-P(v)\|\leq \|u-v\|$. In particular, it is a
uniformly continuous map and its modulus of uniform continuity is $\Delta_P(\epsilon)=\epsilon$.

We start by showing that the class of Hilbert spaces with projections 
has the free amalgamation property:

\begin{lema} Let $(\H_0,P_0)\subset (\H_i,P_i)$ where $i=1,2$ and $H_1
  \ind_{H_0} H_2$ be (possibly finite dimensional) Hilbert spaces with
  projections. Then
  $H_3=\spane\{H_1,H_2\}$ is a Hilbert space and
  $P_3(v_3)=P_1(v_1)+P_2(v_2)$ is a well defined projection, where
  $v_3=v_1+v_2$ and $v_1\in H_1$, $v_2\in H_2$.
  \end{lema}

\begin{proof}
  It is clear that $H_3=\spane\{H_1\cup H_2\}$ is a Hilbert space
  containing
  $H_1$ and $H_2$. It remains to prove
that $P_3$ is a projection map and that it is well defined. We denote
by $Q_0$,
$Q_1$, $Q_2$ the projections onto the spaces $H_0$, $H_1$ and $H_2$ 
respectively. 

Since $H_0 \subset H_1$, we can write $H_1=H_0\oplus (H_1\cap H_0^\perp)$.
Similarly $H_2=H_0\oplus (H_2\cap H_0^\perp)$. Finally, since $H_1
  \ind_{H_0} H_2$,

%\begin{equation}\label{eq3}
$$ H_3=H_0\oplus (H_1\cap H_0^\perp) \oplus (H_2\cap H_0^\perp) $$
%\end{equation}

% $$H_3=H_0\oplus (H_1\cap H_0^\perp) \oplus (H_2\cap H_0^\perp) (1)$$

Let $v_3\in H_3$. Let $u_0=Q_0(v_3)$, $u_1=P_{H_0^\perp \cap
  H_1}(v_3)=Q_1(v_3)-u_0$,
$u_2=P_{H_0^\perp \cap H_2}(v_3)=Q_2(v_3)-u_0$. Then $v_3=u_0+u_1+u_2$. 

As $H_1\cap H_2=H_0$, we can write $v_3$ in many different ways as a
sum of elements in $H_1$ and $H_2$. Given one such expression,
$v_3=v_1+v_2$, with $v_1\in H_1$ and $v_2\in H_2$, it is easy to see that
$P_1(v_1)+P_2(v_2)=P_0(u_0)+P_1(u_1)+P_2(u_2)$, and thus prove that
$P_3$ is well defined.

% Let $w_0=v_1-u_1$ and $t_0=v_2-u_2$.

% \textbf{Claim:} $w_0,t_0\in H_0$. We first note that $w_0\in H_0$. By
% equation~(\ref{eq1}), $P_{H_0^\perp \cap H_1}(v_3)= P_{H_0^\perp \cap
%   H_1}(v_1)=u_1$, so $v_1-u_1\in H_0$. We get $t_0\in H_0$ in a
% similar way.

% Then $P_3(v_3)=P_1(v_1)+P_2(v_2)=P_1(u_1+w_0)+P_2(u_2+t_0)=P_1(u_1)+
% P_1(w_0)+P_2(u_2)+P_2(t_0)=P_1(u_1)+P_0(w_0)+P_0(t_0)+P_2(u_2)=
% P_1(u_1)+P_0(u_0)+P_2(u_2)$.

% Since the expression on the right only depends on $v_3$ and not on the
% decomposition of $v_3$ in terms of $v_1$ and $v_2$, we get that $P_3$
% is a well defined linear map on $H_3$.

% Since for any $v_3\in H_3$,
% $P_3(v_3)=P_1(v_1)+P_2(v_2)=P_1^*(v_1)+P_2^*(v_2)=P_3^*(v_3)$, we get
% that $P_3^*=P_3$.

% Finally, for any $v_3\in H_3$,
% $P_3^2(v_3)=P_3(P_1(v_1)+P_2(v_2))=P_1(P_1(v_1))+P_2(P_2(v_2))
% =P_1^2(v_1)+P_2^2(v_2)=P_1(v_1)+P_2(v_2)=P_3(v_3)$, so $P_3^2=P_3$.

\end{proof}

Let $T^P$ be the theory of Hilbert spaces with a projection. It is
axiomatized by the theory of Hilbert spaces together with the axioms
(1) and (2) that say that $P$ is a projection. 

Consider first the finite dimensional models. Given an $n$-dimensional
Hilbert space $\H_n$, there are only $n+1$ many pairs $(\H_n,P)$,
where $P$ is a projection, modulo isomorphism. They are classified by
the dimension of $P(H)$, which ranges from $0$ to $n$.

In order to characterize the existentially closed models of $T^P$,
note the following facts:

\begin{enumerate}
\item Let $(\H,P)$ be existentially closed, and $(\H_n,P_{n})$ be an
$n$-dimensional Hilbert space with an orthonormal projection with the
property that $P_{n}(\H_n)=\H_n$. Then $(\H,P)\oplus (\H_n,P_{n})$ is
an extension of $(\H,P)$ with $\dim([P\oplus
P_{n}](\H\oplus \H_n))\geq n$. Since $n$ can be chosen as big as we
want  and $(\H,P)$ is existentially closed, $\dim(P(H))=\infty$.
\item Let $(\H,P)$ be existentially closed, and $(\H_n,P_{0})$ be an
  $n$-dimensional Hilbert space with an orthonormal projection such
  that $P_{0}(H_n)=\{0\}$. Then $(\H,P)\oplus
  (\H_n,P_{0})$ is an extension of $(H,P)$ such that
  $\dim(([P\oplus P_{0}](H\oplus \H_n))^\perp)\geq n$. Since $n$ can
  be chosen as big as we want and $(\H,P)$ is existentially closed,
  $\dim(P(H)^\perp)=\infty$.
\end{enumerate}

The theory $T^P_\omega$ extending $T^P$, stating that $P_\omega$ is a
projection and that there are infinitely many pairwise orthonormal
vectors $v$ satisfying $P_{\omega}(v)=v$ and also infinitely many
pairwise orthonormal vectors
$u$ satisfying $P_{\omega}(u)=0$ gives an axiomatization for the the
model companion of $T^P$, which corresponds to the theory of beautiful
pairs of Hilbert spaces. We will now study some properties of
$T^P_\omega$.

Let $(\H,P)\models T^P_\omega$ and for any $v\in H$ let
$d_P(v)=\|v-P(v)\|$.
% (Note that if $v\in B_n(0)$ then $d_P(v)\in
% [0,n]$, so $d_P$ is a predicate in the formalism of continuous
% logic.)
Then $d_P(v)$ measures the distance between $v$ and the subspace
$P(H)$. The distance function $d_P(x)$ is definable in $(\H,P)$. We
will now prove the converse, that is, that we can definably recover
$P$ from $d_P$.

\begin{lema}
  Let $(\H,P)\models T^P_\omega$. For any $v\in H_\omega$ let
  $d_P(v)=\|v-P(v)\|$. Then $P(v)\in dcl(v)$ in the structure
  $(\H,d_P)$.
\end{lema}

\begin{proof}
  Note that $P(v)$ is the unique element $x$ in $P(H)$ satisfying
  $\|v-x\|=d_P(v)$. Thus $P(v)$ is the unique realization of the
  condition $\vphi(x)=\max\{ d_P(x),|\|v-x\|-d_P(v)|\} =0$.
\end{proof}

\begin{prop}
  Let $(\H,P)\models T^P_\omega$. For any $v\in H_\omega$ let
  $d_P(v)=\|v-P(v)\|$. Then the projection function $P(x)$ is
  definable in the structure $(\H,d_P)$
\end{prop}

\begin{proof}
  Let $(\H,P)\models T^P_\omega$ be $\kappa$-saturated for
  $\kappa>\aleph_0$ and let
  $d_P(v)=\|v-P(v)\|$. Since $d_P$ is definable in the structure
  $(\H,P)$, the new structure $(\H,d_P)$ is still
  $\kappa$-saturated. Let
  $\mathcal{G}_P$ be the graph of the function $P$. Then by the
  previous lemma $\mathcal{G}_P$ is type-definable in $(\H,d_P)$ and
  thus by \cite[Proposition 9.24]{BBHU} $P$ is definable in the
  structure $(\H,d_P)$.
\end{proof}

\begin{nota}
  We write $\tp$ for $L$-types, $\tp_P$ for $L_P$-types and $\qftp_P$
  for quantifier free $L_P$-types. We write $\acl_P$ for the
  algebraic closure in the language $L_P$. We
  follow a similar convention for $\dcl_P$.
\end{nota}

\begin{lema}
$T_\omega^P$ has quantifier elimination.
\end{lema}

\begin{proof}
  It suffices to show that quantifier free $L_P$-types determine the
  $L_P$-types. Let $(\H,P)\models T_\omega^P$ and let $\bar a=(a_1,\dots,a_n),
  \bar b=(b_1,\dots,b_n)\in H^n$. Assume that $\qftp_P(\bar
  a)=\qftp_P(\bar b)$. Then
  \[\tp(P(a_1),\dots,P(a_n))=\tp(P(b_1),\dots,P(b_n))\]
  and
  \[
  \tp(a_1-P(a_1),\dots,a_n-P(a_n))=\tp(b_1-P(b_1),\dots,b_n-P(b_n)).\]
  Let $H_0=P(H)$ and
  let $H_1=H_0^\perp$, both are then infinite dimensional Hilbert spaces and
  $\H=\H_0 \oplus \H_1$. Let $f_0\in Aut(\H_0)$ satisfy
  $f_0(P(a_1),\dots,P(a_n))=(P(b_1),\dots,P(b_n))$ and let $f_1\in
  Aut(\H_1)$ be such that
  $f_1(a_1-P(a_1),\dots,a_n-P(a_n))=(b_1-P(b_1),\dots,b_n-P(b_n))$.
  Let $f$ be the automorphism of $\H$ induced by by $f_0$ and $f_1$,
  that is, $f=f_0\oplus f_1$. Then $f\in Aut(\H,P)$ and
  $f(a_1,\dots,a_n)=(b_1,\dots,b_n)$, so $\tp_P(\ba)=\tp_P(\bb)$.
\end{proof}

\noindent
{\bf Characterization of types:}
By the previous lemma, the $L_P$-type of an
$n$-tuple $\ba=(a_1,\dots,a_n)$ inside a structure $(\H,P)\models T_\omega^P$ is
determined by the type of its projections $\tp(P(a_1),\dots,P(a_n),
a_1-P(a_1),\dots,a_n-P(a_n))$. In particular, we may regard $(\H,P)$
as the direct sum of the two independent pure Hilbert spaces
$(P(H),+,0,\bra,\ket)$ and $(P(H)^\perp,+,0,\bra,\ket)$.

We may therefore characterize definable and algebraic closure, as
follows.

\begin{prop}
Let $(\H,P)\models T_\omega^P$ and let $A\subset H$. Then $\dcl_P(A)=\acl_P(A)=\dcl(A\cup P(A))$. 
\end{prop}

We leave the proof to the reader. Another consequence of the
description of types is:

\begin{prop}
The theory $T_\omega^P$ is $\omega$-stable.
\end{prop}

\begin{proof}
  Let $(\H,P)\models T_\omega^P$ be separable and let $A\subset H$ be
  countable. Replacing $(\H,P)$ for $(\H,P)\oplus (\H,P)$ if necessary, we
  may assume that
  
  $P(H)\cap dcl_P(A)^\perp$ is infinite dimensional and that $P(H)^\perp \cap
  dcl_P(A)^\perp$ is infinite dimensional. Thus every $L_p$-type over $A$
  is realized in the structure $(\H,P)$ and $(S_1(A),d)$ is separable.
\end{proof}

\begin{prop}
  Let $(\H,P)\models T_\omega^P$ be a $\kappa$-saturated domain and let $A,B,C\subset H$ be
  small. Then $\tp_P(A/B\cup C)$ does not fork over $C$ if and only if
  $\tp(A\cup P(A)/B\cup P(B) \cup C\cup P(C))$ does not fork over $C\cup P(C)$.
\end{prop}

Again the proof is straightforward.

\section{Continuous random predicates}

We now come to our main theory and to our first set of results.
We study the expansion of a Hilbert space with a distance function
to a subset of $H$. Let $d_N$ be a new unary predicate and let $L_N$
be the language of Hilbert spaces together with $d_N$. We denote the
$L_N$ structures by $(\H,d_N)$, where $d_N\colon \H\to [0,1]$ and we want to
consider the structures where $d_N$ is a distance to a subset of $H$.
Instead of measuring the actual distance to the subset, we truncate
the distance at one. We start by characterizing the functions
$d_N$ corresponding to distances.

\subsection{The basic theory $T_0$}

We denote by $T_0$ the theory of
Hilbert spaces together with the next two axioms (compare with Theorem 9.11 in \cite{BBHU}):

\begin{enumerate}
\item $\sup_x \min\{1\dotminus d_N(x),\inf_y \max\{|d_N(x)-\|x-y\||,d_N(y)\}\}=0$
\item $\sup_x \sup_y [d_N(y)-\|x-y\|-d_N(x)]\leq 0$
\end{enumerate}

We say a point is \emph{black} if $d_N(x)=0$ and \emph{white} if
$d_N(x)=1$. All other points are gray, darker if $d(x)$ is close to
zero and whiter if $d_N(x)$ is close to one. This terminology follows
\cite{Po}. From the second axiom we get that $d_N$ is uniformly
continuous (with modulus of uniform continuity $\Delta(\epsilon)=\epsilon$). Thus we
can apply the tools of continuous model theory to analyze these
structures.

\begin{lema}
  Let $(\H,d)\models T_0$ be $\aleph_0$-saturated and let $N=\{x\in H:
  d_N(x)=0\}$. Then for any $x\in H$, $d_N(x)=\dist(x,N)$.
\end{lema}

\begin{proof}
  Let $v\in H$ and let $w\in N$. Then by the second axiom $d_N(v)\leq \|v-w\|$
  and thus $d_N(v)\leq \dist(v,N)$.

  Now let $v\in H$. If $d_N(v)=1$ there is nothing to prove, so we may
  assume that $d_N(v)<1$. Consider now the set of statements $p(x)$
  given by $d_N(x)=0$, $\|x-v\|=d_N(v)$.

\textbf{Claim} The type $p(x)$ is approximately satisfiable.
  
Let $\varepsilon>0$. We want to show that there is a realization of the
statements $d_N(x)\leq \varepsilon$, $d_N(v)\leq \|x-v\|+\varepsilon$. By the first axiom there is
$w$ such that $d_N(w)\leq \varepsilon$ and $d_N(v)\leq \|v-w\|+\varepsilon$.

Since $(\H,d)$ is $\aleph_0$-saturated, there is $w\in N$ such that
$\|v-w\|=d_N(v)$ as we wanted.
\end{proof}

There are several ages that need to be considered. We fix $r\in [0,1)$
and we consider the class ${\mathcal K}_r$ of all models of $T_0$ such
that $d_N(0)=r$. Note that in all finite dimensional spaces in
${\mathcal K}_r$ we have at least a point $v$ with $d_N(v)=0$.

\begin{nota}
  If $(\H_i,d_N^i)\models T_0$ for $i\in\{0,1\}$, we write $(H_0,d_N^0)\subset
  (H_1,d_N^1)$ if $H_0\subset H_1$ and $d_N^0=d_N^1\upharpoonright_{H_0}$
  (for each sort).
\end{nota}

We will work in ${\mathcal K}_r$. We start with constructing free amalgamations:

\begin{lema} Let $(\H_0,d_N^0)\subset (\H_i,d_N^i)$ where $i=1,2$ and
  $H_1 \ind_{H_0} H_2$ be Hilbert spaces with distance functions, all
  of them in ${\mathcal K}_r$. Let $H_3=\spane\{H_1,H_2\}$ and let
  \[d_N^3(v)=\min\Big\{ \sqrt{d_N^1(P_{H_1}(v))^2+\|P_{H_2\cap H_0^
      \perp}(v)\|^2},\qquad \qquad \qquad \qquad\]
  \[ \qquad \qquad \qquad \qquad \qquad \qquad
  \sqrt{d_N^2(P_{H_2}(v))^2+\|P_{H_1\cap H_0^\perp}(v)\|^2}\Big\}.\] Then
  $(\H_i,d_N^i)\subset (\H_3,d_N^3)$ for $i=1,2$, and $(\H_3,d_N^3)\in
  {\mathcal K}_r$.
\end{lema}

\begin{proof}
  For arbitrary $v\in H_1$, $\sqrt{d_N^1(P_{H_1}(v))^2+\|P_{H_2\cap H_0^
      \perp}(v)\|^2}=d_N^1(v)$. Since $(\H_0,d_N^0)\subset (\H_i,d_N^i)$ we 
      also have\\
  $\sqrt{d_N^2(P_{H_2}(v))^2+\|P_{H_1\cap H_0^
      \perp}(v)\|^2}=\sqrt{d_N^0(P_{H_0}(v))^2+\|P_{H_1\cap H_0^ \perp}(v)\|^2}\geq
  d_N^1(v)$. Similarly, for any $v\in H_2$, $d_N^3(v)=d_N^2(v)$.

  Therefore $(\H_3,d_N^3)\supset (\H_i,d_N^i)$ for $i\in \{1,2\}$. Now we have to
  prove that the function $d_N^3$ that we defined is indeed a distance
  function.

  Geometrically, $d_N^3(v)$ takes the minimum of the distances of $v$
  to the selected black subsets of $H_1$ and $H_2$. That is, the
  random subset of the amalgamation of $(H_1,d_N^1)$ and $(H_2,d_N^2)$ is
  the union of the two random subsets. It is easy to check that
  $(\H_3,d_N^3)\models T_0$. Since each of $(H_1,d_N^1)$, $(H_2,d_N^2)$ belongs
  to ${\mathcal K}_r$, we have $d_N^1(0)=r=d_N^2(0)$ and thus $d_N^3(0)=r$.
\end{proof}

The class ${\mathcal K}_0$ also has the JEP: let
$(\H_1,d_N^1),(\H_2,d_N^2)$ belong to ${\mathcal K}_0$ and assume that $\H_1\perp
\H_2$. Let $N_1=\{v\in H_1:d_N^1(v)=0\}$ and let $N_2=\{v\in H_2:d_N^2(v)=0\}$. Let $\H_3=\spane( \H_1\cup \H_2)$  and let $N_3=N_1\cup N_2\subset H_3$ and finally, let $d_N^3(v)=\dist(v,N_3)$. Then $(H_3,d_N^3)$ is a witnesses of the JEP in ${\mathcal K}_0$.

\begin{lema}
  There is a model $(\H,d_N)\models T_0$ such that $\H$ is a
  $2n$-dimensional Hilbert space and there are orthonormal vectors
  $v_1,\dots,v_n\in H$, $u_1,\dots,u_n\in H$ such that
  $d_N((u_i+v_j)/2)=\sqrt{2}/2$ for $i\leq j$, $d_N(0)=0$ and
  $d_N((u_i+v_j)/2)= 0$ for $i>j$.
\end{lema}

\begin{proof}
  Let $H$ be a Hilbert space of dimension $2n$, and fix some
  orthonormal basis $\langle v_1,\dots,v_n,u_1,\dots,u_n\rangle$ for $H$. Let
  $N=\{(u_i+v_j)/2: i>j\}\cup \{0\}$ and let $d_N(x)=\dist(x,N)$. Then $d_N(0)=0$ and
  $d_N((u_i+v_j)/2)= 0$ for $i>j$. Since
  $\|(u_i+v_j)/2-(u_k+v_j)/2\|=\sqrt{2}/2$ for $i\neq k$ and
  $\|(u_i+v_j)/2-0\|=\sqrt{2}/2$, we get that
  $d_N(u_i+v_j)=\sqrt{2}/2$ for $i\leq j$
\end{proof}

A similar construction can be made in order to get the Lemma with
$d_N(0)=r$ for any $r\in [0,1]$.
In particular, if we fix an infinite cardinal $\kappa$ and we amalgamate
all possible pairs $(H,d)$ in ${\mathcal K}_r$ for $\dim(H)\leq \kappa$, the theory of
the resulting structure will be unstable.

\subsection{The model companion}\label{modelcomp}

\subsubsection{Basic notations}

We now provide the axioms of the model companion of $T_0\cup\{d_N(0)=0\}$.

Call $T_{d0}$ the theory of the structure built out of amalgamating
all separable Hilbert spaces together with a distance function
belonging to the age ${\mathcal K}_0$. Informally speaking,
$T_{d0}=Th(\varinjlim ({\mathcal K}_0))$. We show how to
axiomatize $T_{d0}$.

The idea for the axiomatization of this part (unlike our third
example, in next section) follows the lines of Theorem 2.4 of
\cite{CP}. There are however important differences in the behavior of
algebraic closures and independence, due to the metric
character of our examples.

%For the work that follows below, we expect the reader to
%be familiar with the proof of that Theorem.

Let $(M,d_N)$ in ${\mathcal K}_0$ be an existentially closed structure
and take some extension $(M_1,d_N)\supset (M,d_N)$. Let $\bar x=(x_1,\dots,x_{n+k})$ be
elements in $M_1\setminus M$ and let $z_1,\dots,z_{n+k}$ be their projections
on $M$. Assume that for $i\leq n$ there are
$\bar y=(y_1,\dots,y_{n})$ in $M_1\setminus M$ that satisfy
$d_N(x_i)=\|x_i-y_i\|$ and $d_N(y_i)=0$. Also assume that for $i>n$,
the witnesses for the distances to the black points belong to $M$,
that is, $d_N^2(x_i)=\|x_i-z_i\|^2+d_N^2(z_i)$ for $i>n$. Also, let us assume
that all points in $\bar x,\bar y$ live in a ball of radius $L$ around
the origin. Let $\bar u=(u_1,\dots,u_{n})$ be the projection of
$\bar y=(y_1,\dots,y_n)$ over $M$.

Let $\vphi(\bar x,\bar y,\bar z,\bar u)$ be a formula such that
$\vphi(\bar x,\bar y,\bar z,\bar u)=0$ describes the values of the
inner products between all the elements of the tuples, that is, it
determines the (Hilbert space) geometric locus of the tuple $(\bar
x,\bar y,\bar z,\bar u)$. The statement
$\vphi(\bar x,\bar y,\bar z,\bar u)=0$ expresses the position of the
potentially new points $\bar x$, $\bar y$ with respect to their
projections into a model. Since $d_N(x_i)=\|x_i-y_i\|$ and $d_N(y_i)=0$,
we have $\|x_i-y_j\|\geq \|x_i-y_i\|$ for $j\leq n$, $i\leq n$. Also, for $i>n$,
$d_N^2(x_i)=\|x_i-z_i\|^2+d_N^2(z_i)$, and get
$\|x_i-y_j\|^2\geq \|x_i-z_i\|^2+d_N^2(z_i)$ for $j\leq n$.

 Note that as $(M_1,d_N)\supset (M,d_N)$, for all $z\in M$,
$d_N^2(z)\leq \|z-y_i\|^2=\|z-u_i\|^2+\|y_i-u_i\|^2$ for $i\leq n$. We may also assume
that there is a positive real $\eta_\varphi$ such that $\|x_i-z_i\|\geq \eta_\varphi$ for $i\leq n+k$ and
$\|y_i-u_i\|\geq \eta_\varphi$ for $i\leq n$.

\subsubsection{An informal description of the axioms}

We want to express that for any parameters $\bar z$, $\bar u$ in the
structure\\
\underline{if} we can find realizations $\bar x$, $\bar y$ of
$\varphi(\bar x,\bar y,\bar z,\bar u)=0$ such that for all $w$ and $i\leq
n$, $d_N^2(w)\leq \|w-u_i\|^2+\|u_i-y_i\|^2$, $\|x_i-y_i\|^2\leq
\|x_i-z_i\|^2+d_N^2(z_i)$ for $i\leq n$, $\|x_i-y_j\|^2\geq
\|x_i-z_j\|^2+d_N^2(z_j)$ for $i>n$ and $j\leq n$,\\
\underline{then} there are tuples $\bar x'$,
$\bar y'$ such that $\vphi(\bar x',\bar y',\bar z,\bar u)=0$,
$d_N(y_i')=0$, $d_N(x_i')=\|x_i'-y_i'\|$ for $i\leq n$
and $d_N^2(x_j)=\|x_j-z_j\|^2+d_N^2(z_j)$ for $j>n$.

That is, for any $\bar z$, $\bar u$ in the
structure, if we can find realizations $\bar x$, $\bar y$ of the
Hilbert space locus given by $\varphi$, and we prescribe ``distances'' $d_N$
that do not clash with the $d_N$ information we already had, in such a
way that for $i\leq n$, the $y_i$'s are black and are witnesses for the
distance to the black set for the $x_i$'s, and for $i>n$ the $x_i$'s do not require new witnesses, then
we can actually find arbitrarily close realizations, {\it with the
  prescribed distances}.

The only problem with this idea is that we do not have an implication
in continuous logic. We replace the expression ``$p\to q$'' by a sequence
of approximations indexed by $\varepsilon$.

\subsubsection{The axioms of $T_N$}

\begin{nota}
  Let $\bz,\bu$ be tuples in $M$ and let $x\in M_1$. By $P_{\spane(\bz
    \bu)}(x)$ we mean the projection of $x$ in the space spanned by
  $(\bz,\bu)$.
\end{nota}

For fixed $\varepsilon\in (0,1)$, let $f\colon[0,1]\to [0,1]$ be a continuous function such that
whenever $\vphi(\bar t)<f(\varepsilon)$ and $\vphi(\bar t')=0$, then

\begin{description}
\item[(a)] $\|P_{\spane(\bar z \bar u)}(x_i)-z_i\|<\varepsilon$.
\item[(b)]  $\|P_{\spane(\bar z \bar u)}(y_i)-u_i\|<\varepsilon$.  
\item[(c)] $|\|t_i-t_j\|-\|t_i'-t_j'\||<\varepsilon$ where $\bar t$ is
the concatenation of $\bar x,\bar y,\bar z,\bar u$.
\end{description} 

Choosing $\varepsilon$ small enough, we may assume that 

\begin{description}
\item[(d)] $\|x_i-P_{\spane (\bz \bu)}(x_i)\|\geq \eta_\varphi/2$ for $i\leq n+k$.
\item[(e)] $\|y_i-P_{\spane (\bz \bu)}(y_i)\|\geq \eta_\varphi/2$ for $i\leq n$. 
\end{description}

Let $\delta=2\sqrt{\varepsilon(L+2)}$ and consider the following
axiom $\psi_{\varphi,\varepsilon}$ (which we write as a positive bounded formula for
clarity) where the quantifiers range over a ball of radius $L+1$:

\noindent
$\forall \bar z \forall \bar u \Big(\forall \bar x \forall \bar y \vphi(\bar x,\bar y,\bar
z,\bar u)\geq f(\varepsilon) \lor \exists w \bigvee_{i\leq n}(d_N^2(w)\geq \|w-u_i\|^2+\|y_i-u_i\|^2+\varepsilon^2 )\lor
\bigvee_{i>n,j\leq n} (\|x_i-y_j\|^2\leq \|x_i-z_i\|^2+d_N^2(z_i)+\varepsilon^2)\lor \\ 
\bigvee_{i,j\leq n,j\neq i} (\|x_i-y_j\|\leq \|x_i-y_i\|-\varepsilon) \lor
\bigvee_{i\leq n}(\|x_i-z_i\|^2+d_N^2(z_i)\leq \|x_i-y_i\|^2-\varepsilon^2)$ 
\[\lor \]
$\lor \exists \bar x \exists \bar y \big{[}
(\vphi(\bar x,\bar y,\bar z,\bar u)\leq f(\varepsilon) \land \bigwedge_{i\leq n} d_N(y_i)\leq \delta)
\land\bigwedge_{i\leq n} |d_N(x_i)-\|x_i-y_i\||\leq 2\delta) \land \bigwedge_{i>n}
|d_N^2(x_i)-\|x_i-z_i\|^2-d_N^2(z_i)|\leq 4 \delta L \big{]}\Big)$

Let $T_N$ be the theory $T_0$ together with this scheme of axioms
$\psi_{\varphi,\varepsilon}$ indexed by all Hilbert space geometric locus
formulas $\vphi(\bar
x,\bar y,\bar z,\bar u)=0$ and $\varepsilon\in(0,1)\cap {\Bbb Q}$. The radius of the
ball that contains all elements, $L$, as well as $n$ and $k$ are
determined from the configuration of points described by the formula $\vphi(\bar
x,\bar y,\bar z,\bar u)=0$.

\subsubsection{Existentially closed models of $T_0$}

\begin{theo}
  Assume that $(M,d_N)\models T_0$ is existentially closed. Then $(M,d_N)\models
  T_N$.
\end{theo}

\begin{proof}
  Fix $\varepsilon>0$ and $\varphi$ as above. Let $\bar z\in M^{n+k}$, $\bar u\in
  M^{n}$ and assume that there are $\bar x$, $\bar y$ with $\vphi(\bar
  x,\bar y,\bar z,\bar u)< f(\varepsilon)$ and $d_N^2(w)<
  \|w-u_i\|^2+\|y_i-u_i\|^2+\varepsilon^2$ for all $w\in M$, $\|x_i-y_i\|^2<
  \|x_i-z_i\|^2+d_N^2(z_i)+\varepsilon^2$ for $i\leq n$, 
  $\|x_i-y_j\|> \|x_i-y_i\|-\varepsilon$ for $i,j\leq n$, $i\neq j$,
  $\|x_i-y_j\|^2 >\|x_i-z_i\|^2+d_N^2(z_j)+\varepsilon^2$ for $i>n$, $j\leq n$. Let $\varepsilon'<\varepsilon$ be such that $\vphi(\bar
  x,\bar y,\bar z,\bar u)< f(\varepsilon')$ and 
  \begin{description}
  \item[(f)] $d_N^2(w)< \|w-u_i\|^2+\|y_i-u_i\|^2+\varepsilon'^2$ for all $w\in M$.
  \item[(g1)] $\|x_i-y_i\|^2>\|x_i-z_i\|^2+d_N(z_i)+\varepsilon'^2$ for $i\leq n$.
  \item[(g2)] $\|x_i-y_j\|> \|x_i-y_i\|-\varepsilon'$ for $i,j\leq n$, $i\neq j$
  \item[(h)]  $\|x_i-y_j\|^2 \geq \|x_i-z_i\|^2+d_N^2(z_i)+\varepsilon'^2$ for $i>n$, $j\leq n$. 
  \end{description}
  
  We construct an extension
  $(H,d_N)\supset(M,d_N)$ where the conclusion of the axiom indexed by $\varepsilon'$
  holds. Since $(M,d_N)$ is existentially closed and the conclusion of
  the axiom is true for $(H,d_N)$ replacing $\varepsilon$ for $\varepsilon'<\varepsilon$, then the
  conclusion of the axiom indexed by $\varepsilon$ will hold for $(M,d_N)$.

  So let $H\supset M$ be such that $\dim(H\cap M^\perp)=\infty$. Let
  $a_1,\dots,a_{n+k}$ and $c_1,\dots,c_{n}\in H$ be such that $\tp(\bar a,\bar
  c/\bar z \bar u)=\tp(\bar x,\bar y/\bar z \bar u)$ and ${\bar a \bar
    c}\ind_{\bar z \bar u}{M}$. We can write $a_i=a_i'+z_i'$ and $c_i=c_i'+u_i'$ for some
  $z_i',u_i'\in M$ and $a_i',c_i'\in M^\perp$. By (d) and (e) $\|a_i'\|\geq \eta/2$ for $i\leq n+k$
  and $\|c_i'\|\geq \eta/2$ for $i\leq n$. Now let $\hat c_i =c_i'+u_i'+\delta'
  c_i'/\|c_i'\|$, where $\delta'=\sqrt{2\varepsilon'(L+2)}$.

  Let the black points in $H$ be the ones from $M$ plus the points
  $\hat c_1,\dots,\hat c_n$.  Now we check that the conclusion of the
  axiom $\psi_{\vphi,\varepsilon'}$ holds.

\begin{enumerate}
\item $\vphi(\bar a,\bar c,\bar z,\bar u)\leq f(\varepsilon')$ since $\tp(\bar a,\bar
  c/\bar z \bar u)=\tp(\bar x,\bar y/\bar z \bar u)$.
\item Since $\|c_i-\hat c_i\|\leq \delta'$ and $\hat c_i$ is black we have $d_N(c_i)\leq \delta'$.
\item We check that the distance from $a_i$ to the black set is as prescribed for $i\leq n$. $d_N(a_i)\leq \|a_i-\hat c_i\|\leq \|a_i-c_i\|+\delta'$ for $i\leq n$. 

\noindent  Also, for $i\not= j, i,j\leq n$, using $(g2)$ we prove
  $\|a_i-\hat c_j\|\geq \|a_i-c_j\|-\delta'  \geq \|a_i-c_i\|-\varepsilon'-\delta' \geq \|a_i-c_i\|-2\delta'$. Finally by (a) $\|a_i-P_{M}(a_i)\|^2+d_N^2(P_{M}(a_i))\geq (\|a_i-z_i\|
  -\varepsilon')^2+(d_N(z_i)-\varepsilon')^2\geq \|a_i-z_i\|^2-2L\varepsilon'+\varepsilon'^2+d_N^2(z_i)-2\varepsilon'+\varepsilon'^2$ and by $(g1)$, we get $\|a_i-z_i\|^2-2L\varepsilon'+\varepsilon'^2+d_N^2(z_i)-2\varepsilon'+\varepsilon'^2 \geq
  \|a_i-c_i\|^2 -2L\varepsilon'-2\varepsilon' \geq \|a_i-c_i\|^2-4\delta'^2$.
\item  We check that $d_N(a_i)$ is as desired for $i> n$. Clearly $\|a_j-\hat c_i\|\geq \|a_j-c_i\|-\delta'$, so $\|a_j-\hat c_i\|^2\geq
  \|a_j-c_i\|^2+\delta'^2-2\delta'2L$ and by $(h)$ we get $ \|a_j-c_i\|^2+\delta'^2-4\delta'L \geq \|a_j-z_j\|^2+d_N^2(z_j)-4\delta' L-\varepsilon'^2+\delta'^2\geq
  \|a_j-z_j\|^2+d_N^2(z_j)-4\delta' L$.
\end{enumerate}

It remains to show that $(M,d_N)\subset (H,d_N)$, i.e., the function $d_N$
on $H$ extends the function $d_N$ on $M$ . Since we added the black
points in the ball of radius $L+1$, we only have to check that for any
$w\in M$ in the ball of radius $L+2$, $d_N^2(w)\leq \|w-\hat
c_i\|^2=\|w-u_i'\|^2+\|c_i'+\delta' (c_i'/\|c_i'\|)\|^2$.

But by $(f)$ $d_N^2(w)\leq \|w-u_i\|^2+\|c_i-u_i\|^2+\varepsilon'^2$, so it suffices to show that
$$\|w-u_i\|^2+\|c_i-u_i\|^2+\varepsilon'^2\leq \|w-u_i'\|^2+\|c_i'\|^2+2\delta' \|c_i'\|+\delta'^2$$
 By (a) $\|w-u_i'\|^2\geq (\|w-u_i\|-\varepsilon')^2$ and is enough to prove that
$$\|w-u_i\|^2+\|c_i-u_i\|^2+\varepsilon'^2\leq (\|w-u_i\|-\varepsilon')^2+\|c_i'\|^2+2\delta'
\|c_i'\|+\delta'^2$$ But $(\|w-u_i\|-\varepsilon')^2+\|c_i'\|^2+2\delta'
\|c_i'\|+\delta'^2=\|w-u_i\|^2-2\varepsilon'\|w-u_i\|+\varepsilon'^2+\|c_i'\|^2+2\delta'
\|c_i'\|+\delta'^2$ and $\|c_i-u_i\|^2\leq \|c_i-u_i'\|^2 +2\varepsilon' \|c_i-u_i'\|+\varepsilon'^2=\|c_i'\|^2 +2\varepsilon' \|c_i'\|+\varepsilon'^2$. Thus, after simplifying, we only need to check $2\varepsilon'\|w-u_i\|+\varepsilon'^2 \leq \delta'^2$ which is true since $2\varepsilon' \|w-u_i\| +\varepsilon'^2\leq 2\varepsilon'(2L+2)+\varepsilon'^2\leq 4\varepsilon'(L+2)$.
\end{proof}

\begin{theo}
Assume that $(M,d_N)\models T_N$. Then $(M,d_N)$ is existentially closed.
\end{theo}

\begin{proof}
  Let $(H,d_N)\supset (M,d_N)$ and assume that $(H,d_N)$ is
  $\aleph_0$-saturated. Let $\psi(\bar x, \bar v)$ be a quantifier free
  $L_N$-formula, where $\bar x=(x_1,\dots x_{n+k})$ and
  $\bar v=(v_1,\dots v_{l})$. Suppose that there are
  $a_1,\dots,a_{n+k}\in H\setminus M$ and $e_1,\dots e_l\in M$ such that $(H,d_N)\models
  \psi(\bar a, \bar e)=0$.
  After enlarging the formula $\psi$ if necessary, we may assume that
  $\psi(\bar x, \bar v)=0$ describes the values of $d_N(x_i)$ for
  $i\leq n+k$, the values of $d_N(v_j)$ for $j\leq l$ and the inner products
  between those elements. We may assume that for $i\leq n$  there is
  $\rho>0$ such that $d_N(a_i)-d(a_i,z)\geq 2\rho$ for all $z\in M$ with $d_N(z)\leq
  \rho$. Since $(H,d_N)$ is $\aleph_0$-saturated, there are $c_1,\dots c_n\in H$
  such that $d_N(a_i)=\|a_i-c_i\|$ and $d_N(c_i)=0$. Then $d(c_i,M)\geq \rho$.
  Fix $\varepsilon>0$, $\varepsilon<\rho,1$.  We may also assume that for $i>n$,
  $|d_N^2(a_i)-\|a_i-P_{M}(a_i)\|^2-d_N^2(P_{M}(a_i))|\leq \varepsilon/2$. Also,
  assume that all points mentioned so far live in a ball of radius $L$
  around the origin.
  Let $b_1,\dots,b_{n+k}\in M$ be the projections of $a_1,\dots,a_{n+k}$
  onto $ M$ and let $d_1,\dots,d_{n}\in M$ be the projections of
  $c_1,\dots,c_{n}$ onto $M$. Let $\vphi(\bar x,\bar y,\bar z,\bar
  u)=0$ be an $L$-statement that describes the inner products between the
  elements listed and such that $\vphi(\bar a,\bar c,\bar b,\bar
  d)=0$. Using the axioms we can find $\bar a'$, $\bar c'$ in $M$ such
  that $\vphi(\bar a',\bar c',\bar b,\bar d)\leq f(\varepsilon)$, $d_N(c_i')\leq
  \delta$ for $i\leq n$, $|d_N(a_i')-\|a_i'-c_i'\||\leq \delta$ for $i\leq n$
  and $|d_N^2(a_i)-\|a_i-b_i\|^2-d_N^2(b_i)|\leq 4L\delta$, where $\delta=\sqrt{2\varepsilon(L+2)}$.
Since $\varepsilon>0$ was arbitrary we get $(M,d_n)\models \inf_{x_1}\dots \inf_{x_{n+k}}\psi(\bar x,\bar v)=0$.
\end{proof}

\section{Model theoretic analysis of $T_N$}

We prove three theorems in this section about the theory $T_N$:

\begin{itemize}
\item $T_N$ is not simple,
\item $T_N$ is not even $NTP_2$! (Of course, this implies the
  previous, but we will provide the proof of non-simplicity as well.)
\item $T_N$ is $NSOP_1$. Therefore, in spite of having a tree
  property, our theory is still ``close to being simple'' in the
  precise sense of not having the $SOP_1$ tree property.
\end{itemize}

These results place $T_N$ in a very interesting situation in the
stability hierarchy for continuous logic.

%For the rest of this section we will study the theory $T_N$. 

\begin{nota}
We write $\tp$ for types of elements in the language $L$ and $\tp_N$
for types of elements in the language $L_N$. Similarly we denote by
$\acl_N$ the algebraic closure in the language $L_N$ and by $\acl$
the algebraic closure for pure Hilbert spaces. Recall that for a set
$A$, $\acl(A) =\dcl(A)$, and this corresponds to the closure of the
space spanned by $A$ (Fact \ref{dclh}).
\end{nota}

\begin{obse}
The theory $T_N$ does not have elimination of quantifiers. We use the
characterization of quantifier elimination given in Theorem 8.4.1 from
\cite{Ho}.
Let $H_1$ be a two dimensional Hilbert space, let $\{u_1,u_2\}$ be an
orthonormal basis for $H_1$ and let $N_1=\{0, u_0+\frac{1}{4} u_1\}$
and let $d^1_N(x)=\min\{1,\dist(x,N_1)\}$. Then $(H_1,d^1_N)\models T_0$.  Let
$a=u_0$, $b=u_0-\frac{1}{4}u_1$ and $c=u_0+\frac{1}{4}u_1$. Note that
$d^1_N(b)=\frac{1}{2}$. Let $(H_1',d^1_N)\supset (H_1,d^1_N)$ be existentially closed. Now
let $H_2$ be an infinite dimensional separable Hilbert space and let
$\{v_i:i\in \omega\}$ be an orthonormal basis. Let $N_2=\{x\in H: \|x-v_1\|=\frac{1}{4},
P_{\spane(v_1)}(x)=v_1\}\cup\{0\}$ and let $d^2_N(x)=\min\{1,\dist(x,N_2)\}$. Let
$(H_2',d^2_N)\supset (H_2,d^2_N)$ be existentially closed. Then
$(\spane(a),d^1_N
\upharpoonright_{\spane(a)})\cong(\spane(v_1),d^2_N\upharpoonright_{\spane(v_1)})$
and they can be identified say by a function $F$. But $(H_1',d^1_N)$ and
$(H_2',d^2_N)$ cannot be amalgamated over this common substructure: If
they could, then we would have
$\dist(F(b),v_1+\frac{1}{4}v_i)=\dist(b,v_1+\frac{1}{4}v_i)<\frac{1}{2}$ for some $i>1$ and thus
$d^1_N(b)<\frac{1}{2}$, a contradiction.

In this case, the main reason for this failure of amalgamation resides
in the fact that $(\spane(a),d^1_N
\upharpoonright_{\spane(a)})\cong(\spane(v_1),d^2_N\upharpoonright_{\spane(v_1)})$
is not a model of $T_0$: informally, the distance values around $v_1$ are
determined by an ``external attractor'' (the black point
$u_0+\frac{1}{4}u_1$ or the black ring orthogonal to $v_1$ at distance
$\frac{1}{4}$) that the subspace $(\spane(a),d^1_N
\upharpoonright_{\spane(a)})$ simply cannot see. This violates Axiom
(1) in the description of $T_0$. This ``noise external to the
substructure'' accounts for the failure of amalgamation, and ultimately
for the lack of quantifier elimination.
\end{obse}

In~\cite[Corollary 2.6]{CP}, the authors show that the algebraic
closure of the expansion of a simple structure with a generic subset
corresponds to the algebraic in the original language. However, in our
setting, the new algebraic closure $\acl_N(X)$ does not agree with the
old algebraic closure $\acl(X)$:

\begin{obse}\label{aclN}
  The previous construction shows that $\acl_N$ does not coincide
  with $\acl$. Indeed, $c\in \acl_N(a)\setminus \acl(a)$ - the set of solutions
  of the type $\tp_N(c/a)$ is $\{ c\}$, but $c\notin dcl(a)$ as $c\notin
  \spane(a)$.
\end{obse}

However, models of the basic theory $T_0$ are $L_N$-algebraically
closed. The proof is similar to~\cite[Proposition 2.6(3)]{CP}:

\begin{lema}\label{aclN2}
Let $(M,d_N)\models T_N$ and let $A\subset M$ be such that
$A=\dcl(A)$ and $(A,d_N\upharpoonright_A)\models T_0$. Let $a\in M$. Then $a\in\acl_N(A)$ if and
only if $a\in A$.
\end{lema}

\begin{proof}
  Assume $a\notin A$. We will show that $a\notin \acl_N(A)$. Let $a'\models \tp(a/A)$
  be such that ${a'}\ind_{A}{M}$. Let $(M',d_N)$ be an isomorphic copy
  of $(M,d_N)$ over $A$ through $f:M\to_A M'$ such that $f(a)=a'$. We
  may assume that ${M'}\ind_{A}{M}$. Since $(A,d_N\upharpoonright_A)$ is
  an amalgamation base, $(N,d_N)=(M\oplus_AM',d_N) \models T_0$. Let $(N',d_N)\supset
  (N,d_N)$ be an existentially closed structure. Then
  $\tp_N(a/A)=\tp_N(a'/A)$ and therefore $a\notin \acl_N(A)$.
\end{proof}

As $T_N$ is model complete, the types in the extended language are
determined by the existential formulas within them, i.e. formulas of
the form $\inf_{\bar y} \varphi(\bar y, \bar x)=0$

Another difference with the work of Chatzidakis and Pillay is that the
analogue to~\cite[Proposition 2.5]{CP} no longer holds. Let $a,b,c$ be as in
Observation \ref{aclN}; notice that
$(\spane(a),d_N\upharpoonright_{\spane(a)})\cong(\spane(v_1),d_N
\upharpoonright_{\spane(v_1)})$.
However, $(H'_1,d_N,a)\not\equiv (H'_2,d_N,v_1)$. Instead, we can show the
following weaker version of the Proposition.

\begin{prop}\label{elemequivN}
  Let $(M,d_N)$ and $(N,d_N)$ be models of $T_N$ and let $A$ be a
  common subset of $M$ and $N$ such that
  $(\spane(A),d_N\upharpoonright_{\spane(A)})\models T_0$. Then
  \[ (M,d_N)\equiv_A (N,d_N).\]
\end{prop}

\begin{proof}
  Assume that $M\cap N=\spane(A)$. Since
  $(\spane(A),d_N\upharpoonright_{\spane(A)})\models T_0$, it is an
  amalgamation base and therefore we may consider the free amalgam
  $(M\oplus_{\spane(A)}N,d_N)$ of
  $(M,d_N)$ and $(N,d_N)$ over
  $(\spane(A),d_N\upharpoonright_{\spane(A)})$. Let now $(E,d_N)$ be a model
  of $T_N$ extending $(M\oplus_{\spane(A)}N,d_N)$. By the model
  completeness of $T_N$, we have that $(M,d_N)\prec (E,d_N)$ and $(N,d_N)\prec
  (E,d_N)$ and thus $(M,d_N)\equiv_A (N,d_N)$.
\end{proof}

\subsection{Generic independence}\label{genericindep}
In this section we define an abstract notion of independence and study
its properties.

Fix $(\U,d_N)\models T_N$ be a $\kappa$-universal domain.

\begin{defi}
 Let $A,B,C\subset \U$ be small sets. We say that $A$ is $*$-independent
 from $B$ over $C$ and write $A \ind[*]_C B$ if $\acl_N(A\cup C)$ is
 independent (in the sense of Hilbert spaces) from $\acl_N(C\cup B)$ over
 $\acl_N(C)$. That is, $A \ind[*]_C B$ if for all $a\in \acl_N(A\cup C)$,
 $P_{\overline{B\cup C}}(a)=P_{\overline{C}}(a)$, where $\overline{B\cup
   C}=\acl_N(C\cup B)$ and $\overline{C}=\acl_N(C)$.
\end{defi}

\begin{prop}\label{*independence}
  The relation $\ind[*]$ satisfies the following properties (here $A$,
  $B$, etc., are any small subsets of $\U$):
  \begin{enumerate}
  \item Invariance under automorphisms of $\U$.
  \item Symmetry: $A \ind[*]_C B \Longleftrightarrow B \ind[*]_C A$.
  \item Transitivity: $A \ind[*]_C BD$ if and only if
    $A \ind[*]_C B$ and $A \ind[*]_{BC} D$.
  \item Finite Character: $A \ind[*]_C B$ if and only
    $\ba \ind[*]_C B$ for all $\ba \in A$ finite.
  \item Local Character: If $\bar a$ is any finite tuple, then there
    is countable $B_0 \subseteq B$ such that
    $\bar a \ind[*]_{B_0} B$.
  \item Extension property over models of $T_0$. If $(C,d_N\upharpoonright_C)\models T_0$, then
 we can find $A'$ such that $\tp_N(A/C)=\tp_N(A'/C)$ and $A' \ind[*]_C B$.
 \item Existence over models: $\bar a\ind[*]_MM$ for any $\bar a$.
\item Monotonicity: $\bar a \bar a'\ind[*]_M\bar b \bar b'$ implies $\bar a\ind_M \bar b$.
  \end{enumerate}
\end{prop}

\begin{proof}
\begin{enumerate}
  \item Is clear.
  \item It follows from the fact that independence in Hilbert spaces satisfies  
Symmetry (see Proposition \ref{independence}).
  \item It follows from the fact that independence in Hilbert spaces satisfies  
 Transitivity (see Proposition \ref{independence}).
  \item Clearly $A \ind[*]_C B$ implies that
    $\ba \ind[*]_C B$ for all $\ba \in A$ finite. On the other hand if
    $\ba \ind[*]_C B$ for all $\ba \in A$ finite, then for a dense
    subset $A_0$ of
     $A$, $A_0 \ind[*]_C B$ and thus $A \ind[*]_C B$.
   \item Local Character: let $\bar a$ be a finite tuple. Since
     independence in Hilbert spaces satisfies local character, there
     is $B_1 \subseteq \acl_N(B)$ countable such that $\bar a \ind[*]_{B_1}
     B$. Now let $B_0 \subseteq B$ be countable such that $\acl_N(B_0)\supset
     B_1$. Then  $\bar a \ind[*]_{B_0} B$.
   \item Let $C$ be such that $(C,d_N\upharpoonright_C)\models T_0$. Let $D\supset
     A\cup C$ be such that $(D,d_N\upharpoonright_D)\models T_0$ and let $E\supset B\cup
     C$ be such that $(E,d_N\upharpoonright_E)\models T_0$. Changing $D$ for
     another set $D'$ with $\tp_N(D'/C)=\tp_N(D/C)$, we may assume
     that the space generated by $D'\cup E$ is the free amalgamation of
     $D'$ and $E$ over $C$.
   By lemma \ref{aclN2} $D'$, $E$ are algebraically closed and $D' \ind[*]_C B$.
   \item It follows from the definition of $*$-independence.
   \item It follows from the definition of $*$-independence and transitivity.
\end{enumerate}
\end{proof}

Therefore we have a natural independence notion that satisfies many
good properties, but not enough to guarantee the simplicity of
$T_N$. 

We will show below that the theory $T_N$ has both $TP_2$ and $NSOP_1$. This places it in an
interesting area of the stability hierarchy for continuous model
theory: while having the tree property $TP_2$ and therefore lacking
the good properties of $NTP_2$ theories, it still has a quite
well-behaved independence notion $\ind[*]$, good enough to guarantee
that it does not have the $SOP_1$ tree property. Therefore, although
the theory is not simple, it is reasonably close to this family of 
theories.

\subsection{The failure of simplicity}\label{simplicity}

\begin{theo}
The theory $T_N$ is not simple.
\end{theo}

  The proof's idea uses a characterization of simplicity in terms of
  the number of partial types due to Shelah (see~\cite{Sh}; see
  also Casanovas~\cite{Cas} for further analysis): $T$ is simple iff
  for all $\kappa,\lambda$ such that
  $NT(\kappa,\lambda)<2^\kappa+\lambda^\omega$, where
  $NT(\kappa,\lambda)$ counts the supremum of the cardinalities of
  families of pairwise incompatible partial types of size $\leq
  \kappa$ over a set of cardinality $\leq \lambda$. This holds for
  continuous logic as well. We show that $T_N$ fails this
  criterion.

  \begin{proof}
    
  Fix $\kappa$ an infinite cardinal and $\lambda\geq \kappa$. We will
  find a complete submodel $M_f$ of the monster model, of density
  character $\lambda$, and $\lambda^\kappa$ many types over subsets of
  $M_f$ of power $\kappa$ in such a way that we guarantee that they
  are pairwise incompatible in a uniform way.

  Now also fix some orthonormal basis of $M_f$, listed as 
  \[ \{ b_i|i <\kappa\}\cup \{ a_j|j <\lambda \} \cup \{ c_X |
  X\in P_\kappa(\lambda)\}.\]
 Also fix, \emph{for every} $X\in P_\kappa(\lambda)$, a bijection
 \[ f_X: \{ b_i|i <\kappa\} \to \{
   a_j|j \in X \}.\]
 Let the ``black points'' of $M_f$ consist of the
  set 
  \[ N = \{ c_X+b_i+(1/2)f_X(b_i)\mid i<\kappa,X\in
    P_\kappa(\lambda)\}\cup \{ 0\}\] and as usual
  define $d_N(x)$ as the distance from $x$ to $N$. This is a submodel
  of the monster.

  Let $A_X:=\{ b_i|i <\kappa\} \cup \{ a_j | j\in X\}$ for each $X\in
  P_\kappa(\lambda)$.

  The crux of the proof is to notice that if $X\not= Y$ then the types
  $\tp(c_X/A_X)$ and $\tp(c_Y/A_Y)$ are incompatible, thereby 
  witnessing that there are $\lambda^\kappa$ many incompatible types:

  Suppose there is some $c$ such that $\tp(c/A_X)=\tp(c_X/A_X)$ and
  $\tp(c/A_Y)=\tp(c_Y/A_Y)$. Take (wlog) $j\in Y\setminus X$. Pick
  $\ell<\kappa$ such
  that $f_Y(b_\ell)=a_j$. Let $k\in X$ be such that $f_X(b_\ell)=a_k$.

  In $M_f$, the distance to black of $c_X+b_\ell-\frac{1}{2}a_k$ is
  $1$: by definition,
  $c_X+b_\ell+\frac{1}{2}a_k=c_X+b_\ell+\frac{1}{2}f_X(b_\ell)\in N$
  and the only difference between  $c_X+b_\ell-\frac{1}{2}a_k$ and
  $c_X+b_\ell+\frac{1}{2}a_k$ is the sign in front of an element of an
  orthonormal basis.

  Therefore the distance to black of $d=c+b_\ell-\frac{1}{2}a_k$ is
  also $1$ (in the
  monster). However, $e=c+b_\ell+\frac{1}{2}a_j$ must be a black
  point, since $e'=c_Y+b_\ell+\frac{1}{2}a_j$ is black (by definition
  of $N$ and since $a_j=f_Y(b_\ell)$ and
  $\tp(c/A_Y)=\tp(c_Y/A_Y)$).
  
  On the other hand, the distance from $e$ to $d$ is
  $\frac{\sqrt{2}}{2}<1$. This contradicts that the color of $d$ is
  $1$.
  \end{proof}

This stands in sharp contrast with respect to the result by
Chatzidakis and Pillay in the (discrete) first order case. The
existence of these incompatible types is rendered possible here by the
presence of ``euclidean'' interactions between the elements of the
basis chosen.

So far we have two kinds of expansions of Hilbert spaces by
predicates: either they remain stable (as in the case of the distance
to a Hilbert subspace as in the previous section) or they are not even
simple. 

\subsection{$T_N$ has the tree property $TP_2$}

\begin{theo}
  The theory $T_N$ has the tree property $TP_2$.
\end{theo}

\begin{proof}
 We will construct a complete submodel $M\models T_0$ of the monster model, of density
  character $2^{\aleph_0}$, and a quantifier free formula $\varphi(x;y,z)$ 
 that witnesses $TP_2$ inside $M$. Since this model can be embedded in the monster
 model of $T_N$ preserving the distance to black points, this will show that $T_N$
 has $TP_2$.

  We fix some orthonormal basis of $M$, listed as 
  \[ \{ b_i|i <\omega\}\cup \{ c_{n,i}|n,i <\omega \} \cup \{ a_f | f:\omega\to \omega\}.\]
 Also let the ``black points'' of $M$ consist of the
  set 
  \[ N = \{ a_f+b_n+(1/2)c_{n,f(n)}\mid n<\omega,f:\omega\to \omega\}\cup \{ 0\}\] and as usual
  define $d_N(x)$ as the distance from $x$ to $N$. This is a model of $T_0$ and thus a submodel
  of the monster.

  Let $\varphi(x,y,z)=\max\{1-d_N(x+y-(1/2)z) ,d_N(x+y-(1/2)z)\}$.

  \textbf{Claim 1}
  For each $i$, the conditions $\{\varphi(x,b_i,c_{i,j})=0:j\in \omega\}$ are $2$-inconsistent.
  
  Assume otherwise, so we can find $a$ (in an extension of $M$) such that 
  $d_N(a+b_i+(1/2)c_{i,j})=0$ and $d_N(a+b_i-(1/2)c_{i,l})=1$ for some $j<l$. But 
  then $d(a+b_i+(1/2)c_{i,j},a+b_i-(1/2)c_{i,l})=d((1/2)c_{i,j},-(1/2)c_{i,l})=\sqrt{2}/2<1$.
  Since $a+b_i+(1/2)c_{i,j}$ is a black point, we get that  $d_N(a+b_i-(1/2)c_{i,l})\leq \sqrt{2}/2$
  a contradiction.
  
    \textbf{Claim 2}
  For each $f$ the conditions $\{\varphi(x,b_i,c_{i,f(i)})=0:i\in \omega\}$ are consistent.
  
  Indeed fix $f$ and consider $a_f$, then by construction $d_N(a_f+b_n+(1/2)c_{n,f(n)})=0$
  and $d(a_f+b_n-(1/2)c_{n,f(n)},a_f+b_n+(1/2)c_{n,f(n)})=1$, so $d_N(a_f+b_n-(1/2)c_{n,f(n)})\leq 1$. 
  
  Now we check the distance to the other points in $N$. It is easy to see that
   $d(a_f+b_n-(1/2)c_{n,f(n)},a_f+b_m+(1/2)c_{m,f(m)})>1$ for $m\neq n$,  $d(a_f+b_n-(1/2)c_{n,f(n)},a_g+b_k+(1/2)c_{k,g(k)})>1$ for $g\neq f$ and all indexes $k$.
  Finally, $d(a_f+b_n-(1/2)c_{n,f(n)},0)>1$.
  This shows that $a_f$ is a witness for the claim.

\end{proof}

\subsection{$T_N$ and the property $NSOP_1$}

Chernikov and Ramsey have proved that whenever a first order discrete
theory satisfies the following properties (for arbitrary models and
tuples), then the theory satisfies the $NSOP_1$ property
(see~\cite[Prop. 5.3]{CR}).

\begin{itemize}
\item \emph{Strong finite character:} whenever $\bar a$ depends on $\bar b$ over
  $M$, there is a formula $\varphi(x,\bar b,\bar m)\in \tp(\bar a/\bar b M)$ such that
  every $\bar a'\models\varphi(\bar x,\bar b,\bar m)$ depends on $\bar b$ over $M$.
\item \emph{Existence over models:} $\bar a\ind_M M$ for any $\bar a$.
\item \emph{Monotonicity:} $\bar a \bar a'\ind_M\bar b \bar b'$ implies $\bar a\ind_M \bar b$.
\item \emph{Symmetry:} $\bar a\ind_M \bar b \quad \Longleftrightarrow \bar b\ind_M \bar a$.
\item \emph{Independent amalgamation:} $\bar c_0\ind_M \bar c_1$, $\bar b_0\ind_M \bar c_0$,
  $\bar b_1\ind_M \bar c_1$, $\bar b_0\equiv_M \bar b_1$ implies there exists $\bar b$ with
  $\bar b\equiv_{\bar c_0M}\bar b_0$, $\bar b\equiv_{\bar c_1M}\bar b_1$.
\end{itemize}

We prove next that in $T_N$, $\ind^*$ satisfies analogues of these
five properties - we may thereby conclude that $T_N$ can be regarded
(following the analogy) as a $NSOP_1$ continuous theory.

In what remains of the paper, we prove that $T_N$ satisfies these
properties.

We focus our efforts in \emph{strong finite character} and
\emph{independent amalgamation}, the other properties were proved in 
Proposition \ref{*independence}.

We need the following setting:

Let $\mathbb M$ be the monster model of $T_N$ and $A\subset \mathbb
\mathbb M$. Fix $\mathfrak A$ with $A\subset {\mathfrak A}\subset
\mathbb M$
be such that $\mathfrak A\models T_0$  and let $\bar a=(a_0,\dots,a_n)\in
M$. We say that $(\bar a,A,\mathfrak B)$ is minimal if $\tp({\mathfrak
  B}/A)=\tp({\mathfrak A}/A)$
and for all $\bar b\in {\mathfrak M}$, if $\tp(\bar b/A)=\tp(\bar a/A)$ then
\[ \| \pr_\Bstr(b_0) \|+\cdots+\| \pr_\Bstr(b_n) \|\geq \| \pr_\Bstr(a_0)
  \|+\cdots+\| \pr_\Bstr(a_n) \|.\]
By compactness, for all $p\in S(A)$ there is a minimal $(\bar a,A,\Bstr)$
such that $\bar a \models p$.

Now let $\cl_0(A)$ be the set of all $x$ such that for some minimal
$(\bar a,A,\Bstr)$, $x=\pr_\Bstr(a_0)$ (the first coordinate of $\bar a$). 

\begin{lema}
  If $\tp(\Bstr/A)=\tp(\Astr/A)$ and $x\in \cl_0(A)$ then $x\in \Bstr$.
\end{lema}

\begin{proof}
  Suppose not. Let $\Bstr$ witness this and let $\Cstr$ and $\bar a$ be
  such that $(\bar a,A,\Cstr)$ is minimal and $x=\pr_\Cstr(a_0)$. Since
  $\Cstr\models T_0$, wlog $\Bstr\ind_\Cstr a$
  (independence in the sense of Hilbert spaces). But then
  $\pr_\Bstr(a_i)=\pr_\Bstr(\pr_\Cstr(a_i))$ for each $i$ and thus $\|
  \pr_\Bstr(a_0) \|+\cdots+\| \pr_\Bstr(a_n) \|< \| \pr_\Cstr(a_0)
  \|+\cdots+\| \pr_\Cstr(a_n) \|$. This contradicts minimality.
\end{proof}

A direct consequence of the previous lemma is that  $\cl_0(A)\subset
\bcl_N(A)=\cap_{A\subset \Bstr\models T_N} \Bstr$, as $\cl_0(A)$ belongs to every 
model of the theory $T_N$.

We now define the \emph{essential closure} $\ecl$.
Let
$\cl_{\alpha+1}(A)=\cl_0(\cl_\alpha(A))$ for all ordinals $\alpha$,
$\cl_\delta(A)=\bigcup_{\alpha<\delta}(\cl_\alpha(A))$, and
$\ecl(A)=\bigcup_{\alpha\in On}\cl_\alpha(A)$.

\begin{lema}\label{extensionperp}
  For all $\bar a,B,A$, if $\ecl(A)=A$ then there is $\bar b$ such that
  $\tp(\bar b/A)=\tp(\bar a/A)$ and $\bar b\ind_A B$.
\end{lema}

\begin{proof}
  Choose $\Astr\models T_0$ such that $A\subset \Astr$ and $\bar c$ such
  that $\tp(\bar c/A)=\tp(\bar a/A)$ and $(\bar c,A,\Astr)$ is minimal. Since
  $\cl_0(A)=A$, $\pr_\Astr(c_i)\in A$ for all $i\leq n$
  ($\bar c=(c_0,\dots,c_n)$), i.e. $\bar c\ind_A\Astr$. Now choose $\bar b$ such
  that $\tp(\bar b/\Astr)=\tp(\bar c/\Astr)$ and $\bar b\ind_\Astr B$. Then 
  $\bar b$ is as needed.
\end{proof}

\begin{coro}\label{eclequalsacl}
  $\ecl(A)=\acl_N(A)$.
\end{coro}

\begin{proof}
  Clearly $\acl_N(A)\subset \bcl_N(A)$. On the other hand, assume that 
  $x\notin \acl_N(A)$. Let $\Bstr$ be a model of $T_N$ such that 
  $A\subset \Bstr$. By Lemma \ref{extensionperp}, we may assume that 
  $x\ind_A \Bstr$. Then $x\notin\Bstr$, so $x\notin \bcl(A)$, so $x\notin \ecl(A)$.
\end{proof}

\begin{theo}
Suppose $\ecl(A)=A$, $A\subset B,C$, $B\ind[*]_AC$
(i.e. $\ecl(B)\ind_A\ecl(C)$), $\bar a\ind[*]_AB$, $\bar b\ind[*]_AC$ and
$\tp(\bar a/A)=\tp(\bar b/A)$. Then there is $\bar c$ such that $\tp(\bar c/B)=\tp(\bar a/B)$,
$\tp(\bar c/C)=\tp(\bar b/C)$ and $\bar c\ind[*]_ABC$.  
\end{theo}

\begin{proof}
  Wlog $\ecl(B)=B$ and $\ecl(C)=C$. By Lemma~\ref{extensionperp} we
  can find models $A_0$, $A_1$, $B^*$ and $C^*$ of $T_0$ such that
  $A\bar a\subset A_0$, $A\bar b\subset A_1$, $B\subset B^*$ and $C\subset
  C^*$, such that $B^*\ind[*]_AC^*$, $A_0\ind[*]_AB^*$ and
  $A_1\ind[*]_AC^*$. We can also find models of $T_0$, $A_0^*$,
  $A_1^*$ and $D^*$ such that $A_0B^*\subset A_0^*$, $A_1C^*\subset
  A_1^*$ and $B^*C^*\subset D^*$ and wlog we may assume that $\bar a$ and
  $\bar b$ are chosen so that $A_0^*\ind_{B^*}D^*$, $A_1^*\ind_{C^*}D^*$,
  and that there is an automorphism $F$ of the monster model fixing
  $A$ pointwise such that $F(\bar a)=\bar b$, $F(A_0)=A_1$ and
  $F(A_0^*)\ind_{A_1}A_1^*$. Notice that now
  \[ A_0\ind_AD^*\mbox{ and }A_1\ind_AD^*.\]
  We can now find Hilbert spaces $A^*,A_0^{**},A_1^{**}$ and $E$ such
  that
  \begin{enumerate}[(i)]
  \item $E$ is generated by $D^*A_0^{**}A_1^{**}$,
  \item $A\subset A^*\subset A_0^{**}\cap A_1^{**}$, $B^*\subset
    A_0^{**}$, $C^*\subset A_1^{**}$,
  \item There are Hilbert space isomorphisms $G:A_0^{**}\to A_0^*$ and
    $H:A_1^{**}\to A_1^*$ such that
    \begin{enumerate}[a)]
    \item $F\circ G\restriction A^*=H\restriction A^*$,
    \item $G\restriction B^*=\id_{B^*}$, $H\restriction
      C^*=\id_{C^*}$,
    \item $G\cup \id_{D^*}$ generate an isomorphism
      \[ \langle A_0^{**}D^*\rangle \to \langle A_0^*D^*\rangle\]
    \item $H\cup \id_{D^*}$ generate an isomorphism
      \[ \langle A_1^{**}D^*\rangle \to \langle A_1^*D^*\rangle\]
    \item $F\cup G\cup H$ generate an isomorphism
      \[ \langle A_0^{**}A_1^{**}\rangle \to \langle
        F(A_0^*)A_1^*\rangle .\]
    \end{enumerate}
  \end{enumerate}
  We can find these because non-dividing independence in Hilbert spaces has $3$-existence (the independence theorem holds for types over sets).

  Now we choose the ``black points'' of our model: $a\in E$ is black
  if one of the following holds:
  \begin{enumerate}[(i)]
  \item $a\in A_0^{**}$ and $G(a)$ is black,
  \item $a\in A_0^{**}$ and $H(a)$ is black,
  \item $a\in D^*$ and is black.
  \end{enumerate}
  Then in $E$ we define the ``distance to black'' function simply as
  the real distance. Then in $D^*$ there is no change and $G$ and $H$
  remain isomorphisms after adding this structure; $D^*$, $A_0^*$,
  $A_1^*$ and $F(A_0^*)$ witness this.

  So we can assume that $E$ is a submodel of the monster, and letting
  $\bar c=G^{-1}(a)$, $G$ witnesses that $\tp(\bar c/B)=\tp(\bar a/B)$ and $H$
  witnesses that $\tp(\bar c/C)=\tp(\bar b/C)$. We have already seen that
  $A^*\ind_AD^*$ and thus $\bar c\ind[*]_ABC$.
\end{proof}

\begin{prop}
  Suppose $\bar b\nind[*]_AC$, $A\subset B\cap C$ and (wlog)
  $C=\bcl(C)$. Then there exists a formula $\chi\in\tp(\bar b/C)$ such that
  for all $\bar a\models \chi$, $\bar a\nind[*]_AC$.
\end{prop}

\begin{proof}
  By compactness, we can find $\varepsilon>0$ such that (letting
  $\bar b=(b_0,\dots,b_n)$, 
  ($\bar a=(a_0,\dots,a_n)$),
  \begin{equation}
    \label{eq:forking}
    \forall \bar a\models \tp(\bar b/B), \|
    \pr_C(a_0)\|+\cdots+\|\pr_C(a_n)\| \geq \varepsilon +
  \| \pr_{\bcl(A)}(a_0)\|+\cdots+\|\pr_{\bcl(A)}(a_n)\|.
  \end{equation}
Again by compactness we can find $\chi\in \tp(\bar b/B)$ such
that (\ref{eq:forking}) holds when we replace $\tp(\bar b/B)$ by $\chi$
and $\varepsilon$ by $\varepsilon/2$, that is: 

\begin{equation}
\forall \bar a\models \chi, \|
    \pr_C(a_0)\|+\cdots+\|\pr_C(a_n)\| \geq \varepsilon/2 +
  \| \pr_{\bcl(A)}(a_0)\|+\cdots+\|\pr_{\bcl(A)}(a_n)\|.
  \end{equation}
and in particular $\bar a\nind[*]_AC$, as we wanted
\end{proof}

%Let now $B^*$ be the orthogonal complement of $\bcl(B)$ from
%$\bcl(A)$, i.e.
%\[ B^*=\{ x-\pr_{\bcl(A)}(x)\mid x\in \bcl(B)\} \subset \bcl(B).\]
%Notice that if $f\in Aut(\monst/B)$ then $f$ fixes $B^*$ setwise. Now we
%can choose $b$ so that in addition $b\in B^*$. Then if $a\models
%\tp(b/B)$, $a\in B^*$. So we can choose $\psi\in \tp(b/B)$ so that if
%$a\models \psi$ then
%\[ \| \pr_{B^*}(a_0)\|+\cdots+\|\pr_{B^*}(a_n)\| \geq
%  n+1-\varepsilon^2/1000.\]

\end{document}